\pdfoutput=1
\documentclass[11pt,leqno]{article}

\usepackage{amsmath,amsthm,amssymb}
\usepackage{authblk,subfigure,booktabs,enumerate}
\usepackage{graphicx}
\usepackage{color}
\usepackage{mathabx,bm}
\usepackage{algorithm,algpseudocode}
\usepackage{url}
\usepackage{enumitem}

\usepackage{geometry}
\usepackage{listings}

\algnewcommand\algorithmicinput{\textbf{Input:}}
\algnewcommand\Input{\item[\algorithmicinput]}
\algnewcommand\algorithmicoutput{\textbf{Output:}}
\algnewcommand\Output{\item[\algorithmicoutput]}

\theoremstyle{plain}
\newtheorem{thm}{Theorem}[section]
\newtheorem{prop}[thm]{Proposition}

\newtheorem{lem}[thm]{Lemma}

\theoremstyle{remark}

\theoremstyle{definition}

\newtheorem{defn}[thm]{Definition}

\makeatletter

\numberwithin{equation}{section}
\numberwithin{figure}{section}
\numberwithin{table}{section}
\numberwithin{algorithm}{section}

\title{A kernel-based method for Schr{\"o}dinger bridges}

\author[1]{Yumiharu Nakano\thanks{E-mail: nakano@comp.isct.ac.jp}}
\affil[1]{Department of Mathematical and Computing Science,
            Institute of Science Tokyo}

\date{\today}

\begin{document}

\maketitle

\begin{abstract}
We characterize the Schr{\"o}dinger bridge problems by a family of Mckean-Vlasov stochastic control problems with no terminal time distribution constraint.
In doing so, we use the theory of Hilbert space embeddings of probability measures and then describe
the constraint as penalty terms defined by the maximum mean discrepancy in the control problems.
A sequence of the probability laws of the state processes resulting from $\varepsilon$-optimal controls converges to a unique solution of
the Schr{\"o}dinger's problem under mild conditions on given initial and terminal time distributions and an underlying diffusion process.
We propose a neural SDE based deep learning algorithm for the Mckean-Vlasov stochastic control problems.
Several numerical experiments validate our methods.

\begin{flushleft}
\end{flushleft}
\begin{flushleft}
\end{flushleft}
\end{abstract}





\section{Introduction}\label{sec:1}

Let $\mathbb{W}^d = C([0,1], \mathbb{R}^d)$ be the space of all $\mathbb{R}^d$-valued continuous functions on $[0,1]$.
Denote by $\mathcal{P}(\mathbb{W}^d)$ the totality of Borel probability measures on $\mathbb{W}^d$.
Similarly, denote by $\mathcal{P}(\mathbb{R}^d)$ the set of all Borel probability measures on
$\mathbb{R}^d$.
In this paper, we are concerned with giving an approximate solution to the so-called Schr{\"o}dinger's bridge problem:
for given $P\in\mathcal{P}(\mathbb{W}^d)$ and $\mu_0,\mu_1\in\mathcal{P}(\mathbb{R}^d)$,
find an element that attains
\begin{equation}
\tag{S}
 H^*:=\inf \left\{H(Q|P): Q\in\mathcal{P}(\mathbb{W}^d), \; Q_0=\mu_0, \; Q_1=\mu_1\right\}.
\end{equation}
Here, $H(Q|P)$ is the relative entropy or Kullback-Leibler divergence of $Q\in\mathcal{P}(\mathbb{W}^d)$ with respect to $P$,
defined by
\[
 H(Q|P) =
 \begin{cases}
   \mathbb{E}_Q\left[\log\dfrac{dQ}{dP}\right],  & \text{if}\;\; Q \ll P, \\
  +\infty, &\text{otherwise},
 \end{cases}
\]
where $\mathbb{E}_R$ denotes the expectation with respect to a probability measure $R$ on a measurable space.
Further, $Q_0$ and $Q_1$ denote the marginal distributions at time $0$ and $1$, respectively, i.e.,
$Q_{\ell}(A)=Q(w=\{w(t)\}_{0\le t\le 1}\in\mathbb{W}^d: w(\ell)\in A)$, $A\in\mathcal{B}(\mathbb{R}^d)$, $\ell =0,1$.

We shall briefly describe the background of the problem (S).
The name Schr{\"o}dinger's problem comes from Erwin Schr{\"o}dinger's works
\cite{schr:1931} and \cite{schr:1932}. His aim was to study a transition probability that most likely occurs under
constraints on the initial and terminal time distributions of the empirical measures of $N$-independent Brownian particles.
The law of large numbers tells us that such transition must be a {\it rare} event.
To determine a reasonable could transition probability among these unlikely possibilities, Schr{\"o}dinger used a
particle migration model with space discretization and adopted the {\it maximum entropy principle}.
Then, after taking the continuous limit, he derived a system of partial differential equations for the optimal transition probability,
the so-called {\it Schr{\"o}dinger system} or {\it Schr{\"o}dinger's functional equations} (see Section \ref{sec:3} below).
We refer to Chetrite et al.~\cite{chetrite-etal:2021}, an english translation  of \cite{schr:1931},
for an exposition of Schr{\"o}dinger's original approach.

F{\"o}llmer \cite{fol:1988} discovers Schr{\"o}dinger's problem is nothing but the one of {\it large deviation}.
To be precise, let $X^{(1)}, \ldots, X^{(N)}$ be $N$-independent Brownian motions on $[0,1]$.
By Sanov's theorem (see, e.g., Dembo and Zeitouni \cite{dem-zei:1998}) for the large deviation principles on empirical measures,
the problem of computing the probability that
the initial and terminal time marginals of the empirical measure from $X^{(1)}, \ldots, X^{(N)}$
are given by $\mu_0$ and $\mu_1$ respectively
is nearly equivalent to the problem (S) for a sufficiently large $N$ when $P$ is given by the law of the process
$X_t=X_0 + W_t$, $0\le t\le 1$, where $X_0\sim \mu_0$ and $\{W_t\}$ is a Brownian motion.

The Schr{\"o}dinger problem has developed theoretically in many directions such as
theory of reciprocal processes, time reversal of diffusions, stochastic mechanics,
stochastic control interpretation, and optimal transport problems
(see, Bernstein \cite{ber:1932}, Jamison \cite{jam:1974,jam:1975}, F{\"o}llmer \cite{fol:1985, fol:1986}, Anderson \cite{and:1982},
Dai Pra \cite{dai:1991}, Mikami \cite{mik:2004}, Mikami and Thieullen \cite{mik-thi:2006}, Nagasawa \cite{nag:1989}, Nelson \cite{nel:2020}, and Zambrini \cite{zam:1986}).
We refer to, e.g., Chen et al.~\cite{chen-etal:2021} and L{\'e}onard \cite{leo:2013} for a detailed survey of Schr{\"o}dinger's bridges.

As for practical applications, the Schr{\"o}dinger bridge problem has many ongoing and prominent areas.
In fact, the Schr{\"o}dinger bridge problem can be viewed as an entropic regularization of the so-called optimal transport problem (see \cite{mik:2004}).
So it can be applied to the main applications of optimal transport problems, say, computer vision and traffic flow problems.
Further, it can be used as a type of Markov chain Monte Carlo method since by solving the Schr{\"o}dinger bridge problem, we can generate any number of samples of
a given terminal distribution. Unlike existing methods such as the Metropolis method,
the Schrodinger bridge has already been shown to be effective for multi-modal distributions (see Huang et al.~\cite{hua-etal:2021}).
The {\it diffusion models} have played a key role in the recent success of image generative AI,
which are based on the time reversal of diffusions (see Song et al.~\cite{son-etal:2020} and Ho et al.~\cite{ho-etal:2020}).
Thus the Schr{\"o}dinger problem has an essential connection with the diffusion models.
De Bortoli et al.~\cite{de-etal:2021} applies the Schr{\"o}dinger bridges to generative modeling.

Several numerical methods for the Schr{\"o}dinger problems have been proposed in the literature.
All of them belongs to a class of the iterative proportional fitting. See, e.g., Chen et al.~\cite{chen-etal:2016}, \cite{de-etal:2021},
Pavon et al.~\cite{pav-etal:2021}, and Vargas et al.~\cite{var-etal:2021}. Basically, in the iterative proportional fitting methods, we need to solve
a ``half" bridge problem at each iteration, and this means that the methods can be applied for a limited class of initial and terminal time distributions.
More precisely,  the Hilbert metric based method by \cite{chen-etal:2016} needs the integral evaluations at each iteration, and thus the both $\mu_0$
and $\mu_1$ need to be analytically known. To overcome this problem, \cite{de-etal:2021}, \cite{pav-etal:2021}, and \cite{var-etal:2021} consider
some statistical learning approaches for handling the cases where $\mu_0$ and $\mu_1$ are empirical or easy to sample.
However, the cases where $\mu_0$ is empirical or easy to sample, but $\mu_1$ is analytical and
difficult to sample are still missing, which may appear in transportation planning of crowds.

In the present paper, we aim to propose numerical methods for the problem (S) that can be applied for these missing cases.
To this end, we first characterize (S) with a class of Mckean-Vlasov stochastic control problems.
In doing so, we employ the theory of Hilbert space embeddings of probability measures, as developed in e.g., Sriperumbudur et al.~\cite{sri-etal:2010},
and then describe
the constraint as penalty terms defined by the maximum mean discrepancy in the control problems.
We show that, under mild conditions on $\mu_0$, $\mu_1$ and $P$,
a sequence of the probability laws of the state processes resulting from $\varepsilon$-optimal controls converges to a unique solution of (S).
For numerical solutions for the Mckean-Vlasov stochastic control problems, we propose a deep learning algorithm
based on neural stochastic differential equations (see, e.g., Kidger et al.~\cite{kid-etal:2021a, kid-etal:2021b}).

The present paper is organized as follows:
In Section \ref{sec:2} we review some basic results on the theory of Hilbert space embeddings of probability measures, as well as
give a sufficient condition for which a given kernel-baed metric metrizes the weak topology on $\mathcal{P}(\mathbb{R}^d)$.
In Section \ref{sec:3}, we state our main result of the characterization between (S) and the Mckean-Vlasov control problems, and
describe the numerical methods.
Several numerical experiments are presented in Section \ref{sec:4}.
Section \ref{sec:5} is devoted to a proof of our main theorem.

We close this section by introducing some notation used throughout the paper.
Denote by $x^{\mathsf{T}}$ the transpose of a vector or matrix $x$.
For an open set $A$ in an Euclidean space, we denote by $C(A)$ the space of continuous functions on $A$.
Further, $C_b^m(A)$ stands for the space of all functions on $A$ having bounded and continuous derivatives up to the order $m$.
As usual, we define $C_b^m(A)$ for non-open sets $A$ by extending the definition of continuity and differentiability to the boundary points using appropriate limits
involving elements of $A$.
For a probability measure $\mathbb{Q}$ on a measurable space $(\Omega,\mathcal{F})$ and a random variable $X$ on
$(\Omega,\mathcal{F},\mathbb{Q})$, we denote by $\mathbb{Q}X^{-1}$ the probability law of $X$ under $\mathbb{Q}$.

\section{Hilbert space embeddings of probability measures}\label{sec:2}

As mentioned in Section \ref{sec:1}, our main idea is to use the theory of Hilbert space embeddings of probability measures.
Let $K\in C^1_b(\mathbb{R}^d\times\mathbb{R}^d)$ be a symmetric and strictly positive definite kernel on $\mathbb{R}^d$, i.e.,
$K(x,y)=K(y,x)$ for $x,y\in\mathbb{R}^d$ and for any pairwise distinct $x_1,\ldots,x_N\in \mathbb{R}^d$ and
$\alpha=(\alpha_1,\ldots,\alpha_N)^{\mathsf{T}}\in\mathbb{R}^N\setminus\{0\}$,
\[
 \sum_{j,\ell=1}^N\alpha_j\alpha_{\ell}K(x_j,x_{\ell})>0.
\]
Then there exists a unique Hilbert space $\mathcal{H}\subset C(\mathbb{R}^d)$ such that $K$ is a reproducing kernel on
$\mathcal{H}$ with norm $\|\cdot\|$ (see, e.g., Wendland \cite{wen:2010}).
We consider
\begin{equation}
\label{eq:2.0}
 \gamma_K(\mu,\nu): = \sup_{f\in\mathcal{H}, \, \|f\|\le 1}\left|\int_{\mathbb{R}^d} f d\mu - \int_{\mathbb{R}^d} f d\nu\right|,
 \quad \mu,\nu\in\mathcal{P}(\mathbb{R}^d),
\end{equation}
called the {\it maximum mean discrepancy} (MMD) (see Greton et al.~\cite{gre-etal:2006}).
We assume that $\gamma_K$ defines a metric on $\mathcal{P}(\mathbb{R}^d)$.
In this case, $K$ is called a {\it characteristic kernel}.

One of sufficient conditions for which $K$ is characteristic is that $K$ is
an integrally strictly positive definite kernel, i.e.,
\[
 \int_{\mathbb{R}^d\times\mathbb{R}^d} K(x,y)\mu(dx)\mu(dy) >0
\]
for any finite and non-zero signed Borel measures $\mu$ on $\mathbb{R}^d$
(see Theorem 7 in Sriperumbudur et al.~\cite{sri-etal:2010}).
Examples of integrally strictly positive definite kernels include the Gaussian kernel
$K(x,y)=e^{-\alpha |x-y|^2}$, $x,y\in\mathbb{R}^d$, where $\alpha>0$ is a constant, and
the Mat{\'e}rn kernel $K(x,y)=K_{\alpha}(|x-y|)$, $x,y\in\mathbb{R}^d$, where $K_{\alpha}$ is the modified Bessel function of
order $\alpha>0$.

A main advantage of using $\gamma_K$ rather than the others such as the Prohorov distance,
the total variation distance, or the Wasserstein distance is that it is relatively easy to handle analytically due to its linear structure.
Indeed, by our boundedness assumption, for $\mu,\nu\in\mathcal{P}(\mathbb{R}^d)$
\[
 \gamma_K(\mu,\nu)=\left\|\int_{\mathbb{R}^d}K(\cdot,x)\mu(dx) - \int_{\mathbb{R}^d}K(\cdot,x)\nu(dx)\right\|_{\mathcal{H}},
\]
whence by the reproducing property,
\begin{equation}
\label{eq:2.1}
 \gamma_K(\mu,\nu)^2 = \int_{\mathbb{R}^d\times\mathbb{R}^d}K(x,y)(\mu-\nu)(dx)(\mu-\nu)(dy)
\end{equation}
(see Section 2 in \cite{sri-etal:2010}).

It should be noted that in the cases of Mat{\'e}rn kernel, $\gamma_K$ defined by \eqref{eq:2.0} metrizes
the weak topology on $\mathcal{P}(\mathbb{R}^d)$, whereas in the Gaussian cases this problem remains open
(see again \cite{sri-etal:2010}).
Here we will give an affirmative answer to this open question.
To this end, consider the case where
\begin{enumerate}
\item[(A1)]$K$ is represented as $K(x,y)=\Phi(x-y)$, $x,y\in\mathbb{R}^d$, for some continuous and integrable function $\Phi$ on $\mathbb{R}^d$ such that
 its Fourier transform $\widehat{\Phi}$ is also integrable on $\mathbb{R}^d$ and satisfies $\widehat{\Phi}(\xi)>0$ for any $\xi\in\mathbb{R}^d$.
\end{enumerate}
The following result is a generalization of Theorem 24 in \cite{sri-etal:2010}:
\begin{prop}
\label{prop:2.1}
Suppose that $(A1)$ holds. Then $\gamma=\gamma_K$ is a metric on $\mathcal{P}(\mathbb{R}^d)$ that metrizes the weak topology.
\end{prop}
\begin{proof}
By the condition (A1), we can apply the Fourier inversion formula to obtain
\begin{equation}
\label{eq:5.0}
\Phi(x) = \int_{\mathbb{R}^d}\rho(\xi)e^{i\xi^{\mathsf{T}}x}d\xi, \quad x\in\mathbb{R}^d,
\end{equation}
where $i$ denotes the imaginary unit and $\rho(\xi)=(2\pi)^{-d/2}\widehat{\Phi}(\xi)$, $\xi\in\mathbb{R}^d$.
For any $\mu\in\mathcal{P}(\mathbb{R}^d)$, denote by $\tilde{\mu}$ the characteristic function of $\mu$. Then by \eqref{eq:2.1} and \eqref{eq:5.0}
\begin{equation}
\label{eq:5.1}
\gamma(\mu,\nu)^2 = \int_{\mathbb{R}^d}\int_{\mathbb{R}^d}\int_{\mathbb{R}^d}e^{i\xi^{\mathsf{T}}(x-y)}\rho(\xi)(\mu-\nu)(dx)(\mu-\nu)(dy)d\xi
= \int_{\mathbb{R}^d}|\tilde{\mu}(\xi) - \tilde{\nu}(\xi)|^2 \rho(\xi) d\xi
\end{equation}
for $\mu, \nu\in\mathcal{P}(\mathbb{R}^d)$. 
This shows that $\gamma$ is a metric on $\mathcal{P}(\mathbb{R}^d)$ and whenever a sequence of probability measures 
$\{\mu_n\}$ converges weakly to a probability measure $\mu$ the distance $\gamma(\mu_n,\mu)$ tends to zero.
Suppose conversely that $\gamma(\mu_n,\mu)\to 0$.
Then \eqref{eq:5.1} and the positivity of $\rho$ means that there exists a subsequence $\mu_{n_k}$ such that
$\tilde{\mu}_{n_k}(\xi)\to \tilde{\mu}(\xi)$, $d\xi$-a.e.
By Glivenko's theorem and its proof, that a sequence in $\mathcal{P}(\mathbb{R}^d)$ converges weakly to some probability measure in $\mathcal{P}(\mathbb{R}^d)$ is
equivalent to the almost everywhere convergence of their characteristic functions
(see, e.g., Theorem 2.6.4 in It{\^o} \cite{ito:1984} and Theorem 26.3 in Billingsley \cite{bil:1995}).
Thus we deduce that $\{\mu_{n_k}\}$ converges weakly to $\mu$.

Consequently, we have shown that any subsequence of $\{\mu_n\}$ contains a further subsequence that converges weakly to $\mu$.
Therefore, by Theorem 2.6 in \cite{bil:1999}, we conclude that $\mu_n\to \mu$.
\end{proof}

Let $\mu_1\in\mathcal{P}(\mathbb{R}^d)$ as in Section \ref{sec:1}. Define
\begin{equation}
\label{eq:2.2}
 K_1(x,y) := K(x,y) - \int_{\mathbb{R}^d}K(x,y^{\prime})\mu_1(dy^{\prime}) - \int_{\mathbb{R}^d} K(x^{\prime},y)\mu_1(dx^{\prime}).
\end{equation}
Then, by \eqref{eq:2.1},
\begin{equation}
\label{eq:2.3}
\gamma_K(\mu,\mu_1)^2=\int_{\mathbb{R}^d}\int_{\mathbb{R}^d}K_1(x,y)\mu(dx)\mu(dy)
+ \int_{\mathbb{R}^d\times\mathbb{R}^d} K(x^{\prime},y^{\prime})\mu_1(dx^{\prime})\mu_1(dy^{\prime}).
\end{equation}
If two random variables $X$ and $\tilde{X}$ on some probability space $(\Omega,\mathcal{F},\mathbb{P})$
are mutually independent and both follow $\mu$, then we have
\begin{equation}
\label{eq:2.4}
 \int_{\mathbb{R}^d}\int_{\mathbb{R}^d}K_1(x,y)\mu(dx)\mu(dy) = \mathbb{E}_{\mathbb{P}}\left[K_1(X,\tilde{X})\right].
\end{equation}
Moreover, given IID samples $X_1,\ldots,X_M\sim \mu$ and $Y_1,\ldots,Y_M\sim\mu_1$, an unbiased estimator of $\gamma^2(\mu,\mu_1)$ is given by
\begin{equation}
\label{eq:2.5}
 \bar{\gamma}_K(\mu,\mu_1) := \frac{1}{M(M-1)}\sum_i\sum_{j\neq i} K(X_i,X_j) - \frac{2}{M^2}\sum_{i,j}K(X_i,Y_j) + \frac{1}{M(M-1)}\sum_i\sum_{j\neq i}K(Y_i,Y_j)
\end{equation}
(see \cite{gre-etal:2006}).

\section{Kernel-based approximate bridges}\label{sec:3}

\subsection{Reduction to Mckean-Vlasov stochastic control problems}

Let $(\Omega,\mathcal{F},\mathbb{P})$ be an atomless probability space equipped with a filtration
$\mathbb{F}=\{\mathcal{F}_t\}_{0\le t\le 1}$ satisfying the usual conditions.
Let $\{W_t\}_{0\le t\le 1}$ be an $m$-dimensional standard $\mathbb{F}$-Brownian motion on
$(\Omega,\mathcal{F},\mathbb{P})$.

Recall from Section \ref{sec:1} that $\mu_0$ and $\mu_1$ are assumed to be the initial and terminal distributions in the problem (S).
We will impose the following condition:
\begin{enumerate}
\item[(A2)] $\mu_0$ and $\mu_1$ have positive densities $\rho_0$ and $\rho_1$, respectively. Further, 
\[
 \int_{\mathbb{R}^d}\left\{|x|^2\rho_0(x) + (|x|^2 + \log\rho_1(x))\rho_1(x)\right\}dx <\infty. 
\]
\end{enumerate}
Also, choose an $\mathcal{F}_0$-measurable random variable $\xi\sim \mu_0$. Then by (A2) we have $\mathbb{E}|\xi|^2<\infty$.

Let $b:[0,1]\times\mathbb{R}^d\to\mathbb{R}^d$ and $\sigma:[0,1]\times\mathbb{R}^d\to\mathbb{R}^{d\times m}$ be Borel measurable.
Denote $a=\sigma\sigma^{\mathsf{T}}$. Assume the following:
\begin{enumerate}
 \item[(A3)] The functions $b$ and $\sigma$ are bounded in $[0,1]\times\mathbb{R}^d$ and Lipschitz continuous with respect to  $x$ uniformly in $t$.
 \item[(A4)] For $i,j=1,\ldots,d$,
 the funciotn $a_{ij}\in C^1_b([0,1]\times\mathbb{R}^d)$ and $\partial a_{ij}/\partial x_k$ is H{\"o}lder continuous with respect to both $t$ and $x$. Moreover,
 there exists a positive constant $c_0$ such that
\[
 \xi^{\mathsf{T}}a(t,x)\xi \ge c_0|\xi|^2, \quad t\in [0,1], \;\; x,y\in\mathbb{R}^d.
\]
\end{enumerate}
Under (A2)--(A4), there exists a unique strong solution $\{X_t\}_{0\le t\le 1}$ of the stochastic differential equation (SDE)
\begin{equation}
\label{eq:3.0}
 dX_t = b(t,X_t)dt + \sigma(t,X_t)dW_t, \quad 0\le t\le 1, \;\; X_0=\xi.
\end{equation}
Further, $\{X_t\}$ has a transition density $p(t,x,s,y)$, and so we have
\[
 \mathbb{P}(X_t\in A) = \int_A\int_{\mathbb{R}^d}p(0,x,t,y)\mu_0(dx)dy, \quad A\in\mathcal{B}(\mathbb{R}^d), \;\; t\in [0,1]
\]
(see, e.g., Karatzas and Shreve \cite[Chapter 5]{kar-shr:1991}).
Moreover, there exist positive constants $C_1\ge 1$ and $c_1\le 1$ such that
\begin{equation}
\label{eq:3.1}
 \frac{1}{C_1t^{d/2}}\exp\left(-\frac{c_1|x-y|^2}{2t}\right)\le p(0,x,t,y)\le 
 \frac{C_1}{t^{d/2}}\exp\left(-\frac{c_1|x-y|^2}{2t}\right), \quad t\in (0,1], \;\; x,y\in\mathbb{R}^d
\end{equation}
(see Aronson \cite{aro:1967}).  In particular, the function $(x,y)\mapsto p(0,x,1,y)$ is positive and bounded on $\mathbb{R}^d\times\mathbb{R}^d$.

We will work in the situation where the prior measure $P$ in the problem (S) satisfies
\begin{equation}
\tag{A5}
 P = \mathbb{P}X^{-1}.
\end{equation}
The conditions (A2) and \eqref{eq:3.1} imply that there exists a unique pair $(\mu_0^*,\mu_1^*)$ of $\sigma$-finite measures such that the so-called
{\it Schr{\"o}dinger system} or {\it Schr{\"o}dinger's functional equation}
\begin{equation}
\label{eq:3.2}
\begin{aligned}
 &\mu_0^*(dx)\int_{\mathbb{R}^d}p(0,x,1,y)\mu_1^*(dy) = \mu_0(dx), \\
 &\mu_1^*(dy) \int_{\mathbb{R}^d} p(0,x,1,y)\mu_0^*(dx) = \mu_1(dy)
\end{aligned}
\end{equation}
holds and $\mu_0^*\sim\mu_0$, $\mu_1^*\sim\mu_1$. 
See \cite{jam:1974} and Nutz \cite{nut:2022} for a proof. See also Beurling \cite{beu:1960} and L{\'e}onard \cite{leo:2019}.
The uniqueness here is understood up to the transformation $(\mu_0^*,\mu_1^*)\mapsto (\kappa\mu_0^*,\kappa^{-1}\mu_1^*)$ for any $\kappa>0$.

Our approach is to reduce the problem (S) to a stochastic control problem where the constraint of terminal time distributions is described by
the kernel-based metric on the probability measures discussed in Section \ref{sec:2}.
We refer to, e.g., \cite{leo:2013} and \cite{chen-etal:2021} for a formal connection between the Schr{\"o}dinger's bridges and stochastic control problems.
Here let us introduce a weak formulation of stochastic control problems as described in Fleming and Soner \cite{fle-son:2006} and
Yong and Zhou \cite{yon-zho:1999}.

Let $\mathcal{U}$ be the set of all $\mathbb{R}^m$-valued Borel measurable function $u$ on $[0,1]\times\mathbb{R}^d$ such that
$u\in C([0,1)\times\mathbb{R}^d)$,
\begin{equation}
\label{eq:3.3a}
 \int_0^1|u(t,X_t)|^2dt < \infty, \quad\mathbb{P}\text{-a.s.},
\end{equation}
and
\begin{equation}
\label{eq:3.3b}
\mathbb{E}\left[\exp\left(\int_0^1u(t,X_t)^{\mathsf{T}}dW_t-\frac{1}{2}\int_0^1|u(t,X_t)|^2dt\right)\right] = 1,
\end{equation}
where we have denoted $\mathbb{E}=\mathbb{E}_{\mathbb{P}}$.
We call $u\in\mathcal{U}$ an \textit{admissible control function}.
Then for any admissible control function $u\in\mathcal{U}$, by Girsanov-Maruyama's theorem, the process
\[
 B_t^u := W_t - \int_0^tu(t,X_t)dt, \quad 0\le t\le 1,
\]
is an $\mathbb{F}$-Brownian motion under the probability measure $\mathbb{Q}^u$ on $(\Omega,\mathcal{F})$ defined by
\[
 \frac{d\mathbb{Q}^u}{d\mathbb{P}} = \exp\left(\int_0^1u(t,X_t)^{\mathsf{T}}dW_t - \frac{1}{2}\int_0^1|u(t,X_t)|^2dt\right).
\]
This means that $(X, B^u, \Omega,\mathcal{F},\mathbb{F},\mathbb{Q}^u)$ is a weak solution of the controlled SDE
\[
 dY_t = (b(t,Y_t) + \sigma(t,Y_t)u(t,X_t))dt + \sigma(t,Y_t)dW_t.
\]

Next, recall from Section \ref{sec:2} that the terminal time distribution constraint can be characterized by the metric $\gamma$, i.e.,
for $u\in\mathcal{U}$, the law $\mathbb{Q}^u(X_1)^{-1}$ coincides with $\mu_1$ if and only if
$\gamma(\mathbb{Q}^u(X_1)^{-1},\mu_1)=0$. In this case, by \eqref{eq:2.3} and \eqref{eq:2.4}, $\gamma(\mathbb{Q}^u(X_1)^{-1},\mu_1)$ is represented as
\[
 \gamma(\mathbb{Q}^u(X_1)^{-1},\mu_1)^2 = \mathbb{E}_{\mathbb{Q}^u}\left[K_1(X_1, \tilde{X}_1^u)\right]
 + \int_{\mathbb{R}^d\times\mathbb{R}^d} K(x^{\prime},y^{\prime})\mu_1(dx^{\prime})\mu_1(dy^{\prime}),
\]
where $K_1$ is as in \eqref{eq:2.2}, and $\tilde{X}_1^u$ denotes an independent copy of $X_1$ under $\mathbb{Q}^u$.

This leads to the following Mckean-Vlasov type stochastic control problem: for $\lambda>0$
\begin{equation}
\tag{C$_{\lambda}$}
 J^*_{\lambda}:= \inf_{u\in\mathcal{U}} J_{\lambda}(u),
\end{equation}
where for $u\in\mathcal{U}$,
\[
 J_{\lambda}(u)=\mathbb{E}_{\mathbb{Q}^u}\left[\int_0^1|u(t,X_t)|^2dt\right] + \lambda  \gamma(\mathbb{Q}^u(X_1)^{-1},\mu_1)^2.
\]
Let $\{\varepsilon_n\}_{n=1}^{\infty}$ and $\{\lambda_n\}_{n=1}^{\infty}$ be positive sequences such that
\[
 \varepsilon_n \searrow 0, \quad \lambda_n\nearrow + \infty, \quad n\to\infty.
\]
Then for each $n\ge 1$ choose $\varepsilon_n$-optimal $u_n\in\mathcal{U}$ for the problem (C$_{\lambda_n}$), i.e., take $u_n\in\mathcal{U}$ such that
\[
 J_{\lambda_n}(u_n)\le J_{\lambda_n}^* + \varepsilon_n.
\]

Here is our main result.
\begin{thm}
\label{thm:3.1}
Suppose that $(A1)$--$(A5)$ hold. Then $H^*<\infty$ and
\[
 H^* = \frac{1}{2}\sup_{\lambda>0}\inf_{u\in\mathcal{U}}J_{\lambda}(u).
\]
Moreover, with the sequence $\{u_n\}$ we have
\begin{align}
\label{eq:3.5}
 &\lim_{n\to \infty}\sqrt{\lambda_n}\gamma(\mathbb{Q}^{u_n}(X_1)^{-1},\mu_1)=0, \\
\label{eq:3.6}
 &\frac{1}{2}\lim_{n\to \infty} J_{\lambda_n}(u_n) = H^*.
\end{align}
\end{thm}
A proof of this theorem is given in Section \ref{sec:5}.

\subsection{Neural SDEs-based methods}

We shall present a numerical method based on neural networks for solving the Mckean--Vlasov control problem (C$_{\lambda}$).
Since the control class $\mathcal{U}$ is given by a subset of continuous functions on $\mathbb{R}^{1+d}$, a natural idea is to replace
$\mathcal{U}$ by a class $\{u_{\theta}\}_{\theta\in\Theta}$ of deep neural networks, and then to solve
(C$_{\lambda}$) with $\mathcal{U}$ replaced by $\{u_{\theta}\}_{\theta\in\Theta}$.
For example, $u_{\theta}$ can be given by a multilayer perceptron with input layer $g_0$, $L-1$ hidden layer $g_1,\ldots,g_{L-1}$, and output layer $g_L$,
where $L\ge 1$ and for $\xi\in\mathbb{R}^{1+m}$,
\begin{align*}
 g_0(\xi) &= \xi, \\
 g_{\ell}(\xi) &= \phi_{\ell-1}(w_{\ell}g_{\ell-1}(\xi) + \beta_{\ell})\in\mathbb{R}^{m_{\ell}}, \quad \ell = 1,\ldots, L
\end{align*}
for some matrices $w_{\ell}$ and vectors $\beta_{\ell}$, $\ell=1,\ldots,L$.
Here $m_{\ell}$ denotes the number of units in the layer $\ell$, and $\phi_{\ell-1}$ is an activation function.
Then the parameter $\theta$ is described by $\theta = (w_{\ell}, \beta_{\ell})_{\ell=1,\ldots,L}$ and
$u_{\theta}(t,x)=g_L(t,x)$, $(t,x)\in [0,1]\times\mathbb{R}^d$.
We refer to, e.g., Bishop \cite{bis:2006} and Goodfellow et al.~\cite{goo-etal:2016} for an introduction to neural networks and deep learning.

Assume here that $\{u_{\theta}\}_{\theta\in\Theta}\subset C_b^1([0,1]\times\mathbb{R}^d)$.
Then, $u_{\theta}$ satisfies \eqref{eq:3.3a} and Novikov's condition, and so \eqref{eq:3.3b} for every $\theta\in\Theta$, whence $\{u_{\theta}\}\subset\mathcal{U}$.
Moreover, since $\sigma u_{\theta}$ satisfies the linearly growth condition and is Lipschitz continuous in $x$ uniformly in $t$,
there exists a unique strong solution $\{X_t^{(\theta)}\}_{0\le t\le 1}$ of
\begin{equation}
\label{eq:3.7}
 dX_t^{(\theta)} = (b(t,X_t^{(\theta)}) + \sigma(t,X_t^{(\theta)})u_{\theta}(t,X_t^{(\theta)}))dt + \sigma(t,X_t^{(\theta)})dW_t, \quad X_0^{(\theta)}=\xi
\end{equation}
on $(\Omega,\mathcal{F},\mathbb{F},\mathbb{P})$.
Hence,
\[
 J_{\lambda}(u) = \mathbb{E}\left[\int_0^1|u_{\theta}(t,X^{(\theta)}_t)|^2dt + \lambda K_1(X^{(\theta)}_1,\tilde{X}_1^{(\theta)})\right] 
  + \lambda \int_{\mathbb{R}^d\times\mathbb{R}^d} K(x^{\prime},y^{\prime})\mu_1(dx^{\prime})\mu_1(dy^{\prime}). 
\]

Consequently, in view of Theorem \ref{thm:3.1}, a natural proxy for (S) is the following stochastic optimization problem:
minimize
\begin{equation}
\label{eq:3.8}
 F(\theta):=\frac{1}{\lambda}\mathbb{E}[u_{\theta}(\tau, X^{(\theta)}_{\tau})] + \mathbb{E}[K_1(X_1^{(\theta)},\tilde{X}_1^{(\theta)})]
\end{equation}
over $\theta\in\Theta$, with suitable $\lambda>0$.
Here $\tau$ follows the uniform distribution on $\{t_i\}_{i=0}^N$ that is independent of $X^{(\theta)}$,
where $\{t_i\}_{i=0}^N$ is a set of time discretized points such that $0=t_0<t_1<\cdots <t_N=1$.

A pseudo code of our algorithm can be described in Algorithm \ref{alg:3.1}.
\begin{algorithm}
\caption{Schr{\"o}dinger bridge with MMD}
\label{alg:3.1}
\begin{algorithmic}[1]
\Input Number $N$ of time steps, the time discretized points set $\{t_i\}_{i=0}^N$, the number $\nu$ of the iterations, the batch size $M_x$ for the spatial variable,
the batch size $M_t$ for the time variable, weight parameter $\lambda>0$
\Output A function $u(t,x)$
\State $X_{0,1},\ldots, X_{0, M_x}, \tilde{X}_{0,1},\ldots, \tilde{X}_{0,M_x}$ $\gets$ IID samples with
      distribution $\mu_0$.
\For{$k=1,2,\ldots,\nu$}
      \State $\{X_{t_i,1}^{(\theta)}\}_{i=0}^n,\ldots, \{X_{t_i,M_x}^{(\theta)}\}_{i=0}^n$, $\{\tilde{X}_{t_i,1}^{(\theta)}\}_{i=0}^n, \ldots, \{\tilde{X}_{t_i,M_x}^{(\theta)}\}_{i=0}^n$ $\gets$
       IID samples of the solution of \eqref{eq:3.7} on $\{t_i\}_{i=0}^N$.
      \State $\tau_1,\ldots,\tau_{M_t}$ $\gets$ IID samples with uniform distribution on $\{t_i\}_{i=0}^N$.
      \State $\widebar{F}(\theta)$ $\gets$ the Monte Carlo estimates of $F(\theta)$ in \eqref{eq:3.8} using $\{X_{t_i,j}^{(\theta)}\}$, $\{\tilde{X}_{t_i,j}^{(\theta)}\}$,
      and $\tau_{\ell}$,
      $j=1,\ldots,M_x$, $\ell=1,\ldots, M_t$.
      \State Take the gradient step on $\nabla_{\theta}\widebar{F}(\theta)$.
\EndFor
\end{algorithmic}
\end{algorithm}

In the case where we are given IID samples $Y_1,\ldots,Y_M$ of $\mu_1$ and use the estimator given by \eqref{eq:2.5}, the term
$\mathbb{E}[\widehat{K}_1(X_1^{(\theta)},\tilde{X}_1^{(\theta)})]$ can be replaced by
\[
 \frac{1}{M(M-1)}\sum_i\sum_{j\neq i}K(X_{1,i}^{(\theta)},X_{1,j}^{(\theta)}) - \frac{2}{M^2}\sum_{i,j}K(X_{1,i}^{(\theta)}, Y_j),
\]
where $\{X_{1,j}^{(\theta)}\}$ is an IID sequence with distribution $\mathbb{P}(X_1^{(\theta)})^{-1}$.
Thus $F(\theta)$ is modified as
\begin{equation}
\label{eq:3.10}
 F_1(\theta) = \frac{1}{\lambda}\mathbb{E}[u_{\theta}(\tau, X^{(\theta)}_{\tau})]
  + \frac{1}{M(M-1)}\sum_i\sum_{j\neq i}K(X_{1,i}^{(\theta)},X_{1,j}^{(\theta)}) - \frac{2}{M^2}\sum_{i,j}K(X_{1,i}^{(\theta)}, Y_j).
\end{equation}
Accordingly, the pseudo code of Algorithm \ref{alg:3.1} can be modified as in Algorithm \ref{alg:3.2}.
\begin{algorithm}
\caption{Schr{\"o}dinger bridge with empirical MMD}
\label{alg:3.2}
\begin{algorithmic}[1]
\Input Number $N$ of time steps, the time discretized points set $\{t_i\}_{i=0}^N$, the number $\nu$ of the iterations, the batch size $M_x$ for the spatial variable,
the batch size $M_t$ for the time variable, weight parameter $\lambda>0$
\Output A function $u(t,x)$
\State $X_{0,1},\ldots, X_{0, M_x}$ $\gets$ IID samples with
      common distribution $\mu_0$.
\For{$k=1,2,\ldots,\nu$}
      \State $\{X_{t_i,1}^{(\theta)}\}_{i=0}^n,\ldots, \{X_{t_i,M_x}^{(\theta)}\}_{i=0}^n$ $\gets$
       IID samples of the solution of \eqref{eq:3.7} on $\{t_i\}_{i=0}^N$.
      \State $\tau_1,\ldots,\tau_{M_t}$ $\gets$ IID samples with uniform distribution on $\{t_i\}_{i=0}^N$.
      \State $\widebar{F}_1(\theta)$ $\gets$ the Monte Carlo estimates of $F_1(\theta)$ in \eqref{eq:3.10} using $\{X_{t_i,j}^{(\theta)}\}$, $\{\tilde{X}_{t_i,j}^{(\theta)}\}$,
      and $\tau_{\ell}$,
      $j=1,\ldots,M_x$, $\ell=1,\ldots, M_t$.
      \State Take the gradient step on $\nabla_{\theta}\widebar{F}_1(\theta)$.
\EndFor
\end{algorithmic}
\end{algorithm}

\section{Numerical experiments}\label{sec:4}

Here we test our two algorithms through several numerical experiments.
All of numerical examples below are implemented in PyTorch on a NVIDIA Tesla P100 GPU with 256GB memory.
In particular, we use the library \texttt{torchsde} (Li \cite{li:2020}) in PyTorch for an SDE solver and the back propagation of neural SDEs.
Part of the implementation used in the numerical experiments is available at
\url{https://github.com/yumiharu-nakano/kernelSB}.

\subsection{1D Bridge: from a Dirac measure to a bimodal distribution}\label{sec:4.1}

First we examine the one-dimensional case where
$\mu_0(dx)=\delta_0(dx)$ and $\mu_1(dx)=\rho_1(x)dx$ with
\[
\rho_1(x) = \frac{1}{2\sqrt{\pi}}\left(e^{-(x+1)^2} + e^{-(x-1)^2}\right), \quad x\in\mathbb{R}.
\]
Consider a standard Brownian motion as the underlying process $\{X_t\}$, i.e., take $b=0$ and $\sigma=1$.
Since the analytical form of the target distribution is available,
we adopt Algorithm \ref{alg:3.1} with the Gaussian kernel $K(x,y)=e^{-|x-y|^2}$ and the penalty parameter $\lambda$ defined by $1/\lambda=0.005$.
The number $N$ of the time steps is $256$, and the batch sizes $M_x$ and $M_t$ for the spatial variable and time variable are both set to be $128$.
For each $\theta$, the control function $u_{\theta}(t,x)$ is described by a multi-layer perceptron with 2 hidden layer (641 parameters).
In the stochastic optimization, we choose the well-known Adam optimizer with learning rate $0.001$, and stop the algorithm after $5000$ iterations.
Figure \ref{fig:4.1} shows the time series of the histograms of the resulting state process $\{X_t^{(\theta)}\}$ with respect to
an optimized control function $u_{\theta}$, and indicates a satisfactory result.
\begin{figure}[htbp]
\centering
\includegraphics[width=1.0\columnwidth, bb = 0 0 1152 216]{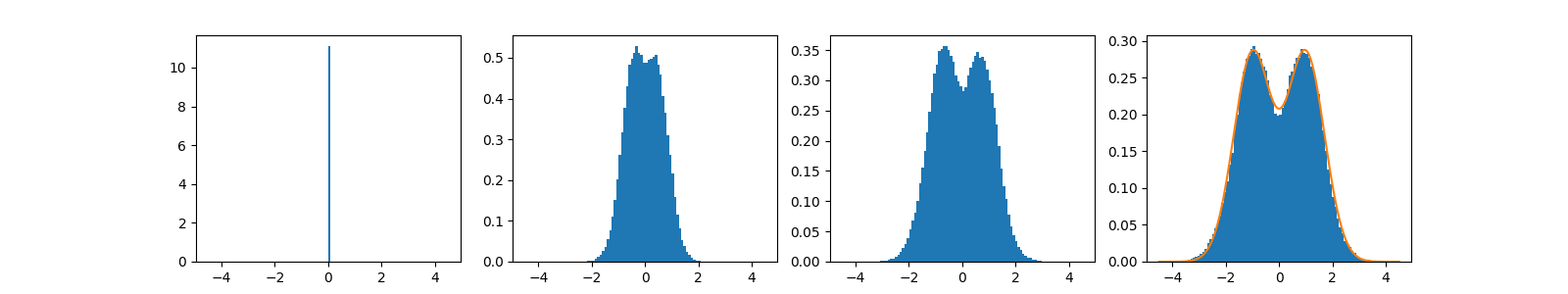}
\caption{Histograms of optimized $X_t^{(\theta)}$'s for $t=0, 0.33, 0.67, 1$ from left to right. Generated with $2\times 10^5$ samples.
The true density $\rho_1$ is plotted in orange.}
\label{fig:4.1}
\end{figure}
Figure \ref{fig:4.2} presents the resulting learning curves for $1/\lambda = 0.5, 0,05, 0.005, 0.0005$, where the other parameters remain unchanged.
Recall that $F(\theta)$ is our objective function for the stochastic optimization and the relation
$H^*=\lim_{\lambda\to\infty}J_{\lambda}^*$. Thus if $\lambda (F(\theta) + c)$ approaches to $H^*$, then we can say that the optimization gets a desirable gain,
where $c$ is the second term of the right-hand side in \eqref{eq:2.3}.
In this simple case, we can show that the optimal value $H^*$ and the constant $c$ are nearly equal to $0.09651$ and $0.72954$, respectively.
On the other hand, the minimum of the plotted losses $F(\theta)$ is around $-0.4$. This suggests that the actually obtained gain is limited and
the algorithm \ref{alg:3.1} falls into a local minimum for each $\lambda$.
It is worth noting that nevertheless the optimized $X_1^{(\theta)}$ samples $\mu$ correctly, as seen in Figure \ref{fig:4.1}.
\begin{figure}[htbp]
\centering
\includegraphics[width=0.6\columnwidth, bb = 0 0 576 324]{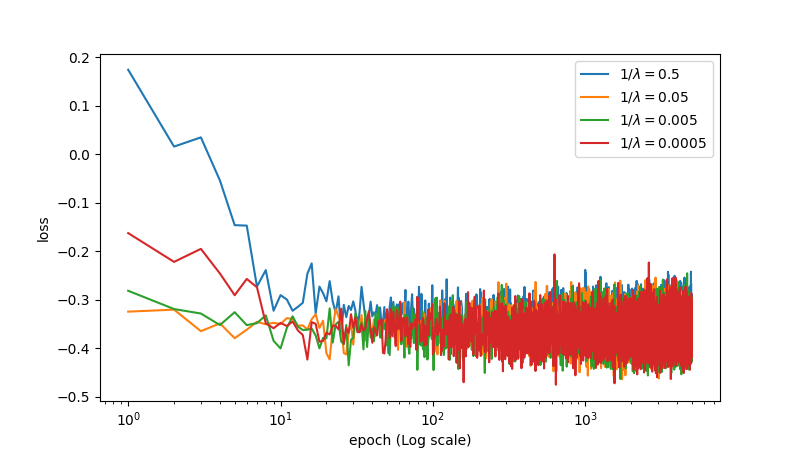}
\caption{The learning curves for $1/\lambda=0.5, 0.05, 0.005, 0.0005$. }
\label{fig:4.2}
\end{figure}

\subsection{2D Interpolation of two datasets}\label{sec:4.2}

Next we examine Algorithm \ref{alg:3.2} for the interpolation of two toy dataset distributions in $\mathbb{R}^2$, both consisting of $1000$ points.
As in the previous example, the Gaussian kernel $K(x,y)=e^{-|x-y|^2}$ and the Adam optimizer with learning rate $0.001$ are used.
Here, the control function $u_{\theta}(t,x)$ is described by a multi-layer perceptron with 3 hidden layer (11362 parameters).
We set $N=256$ and $M_t=64$.

Figure \ref{fig:4.3} shows that the resulting time series of the scatter plots of the solution of the optimized neural SDEs, in the case of
$b=0$, $\sigma=0.05I_2$, $1/\lambda=5\times 10^{-6}$, and $10000$ epochs, where no post-processing, such as kernel density estimation, is performed.
Here we have denoted by $I_2$ the two dimensional identity matrix.
We can see how the circular dataset is continuously transported to the double crescent shaped dataset.
Although a few of the transported points distribute outside of the support of the target distribution, we can say that the result is generally successful.
\begin{figure}[htbp]
\centering
\includegraphics[width=1.0\columnwidth, bb = 0 0 1152 216]{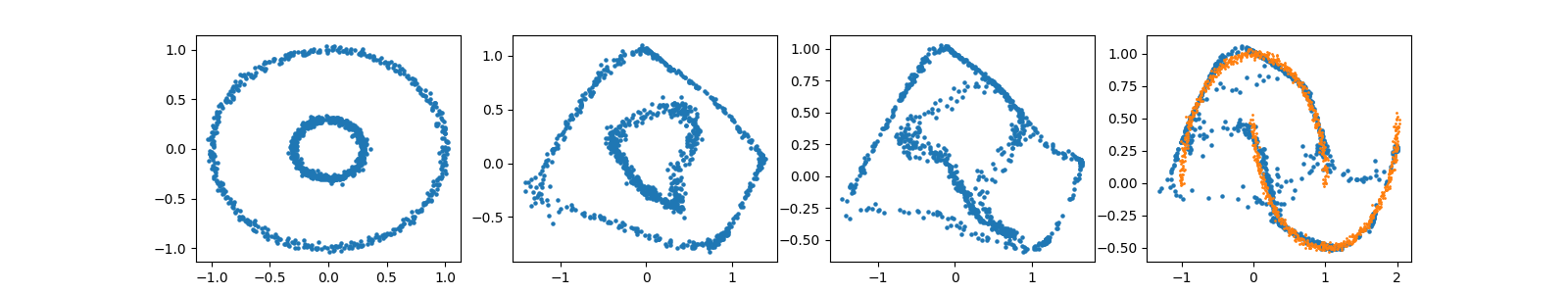}
\caption{Interpolation of a dataset distribution of two circles and one of a double crescent. The scatter plot of optimized $X_t^{(\theta)}$'s for $t=0, 0.33, 0.67, 1$
from left to right. The target distribution is plotted in orange.}
\label{fig:4.3}
\end{figure}
As in Section \ref{sec:4.1}, Figure \ref{fig:4.2} exhibits the resulting learning curves for $1/\lambda=5\times 10^{-2}, 5\times 10^{-4}, 5\times 10^{-6}, 5\times 10^{-8}$.
\begin{figure}[htbp]
\centering
\includegraphics[width=0.6\columnwidth, bb = 0 0 576 324]{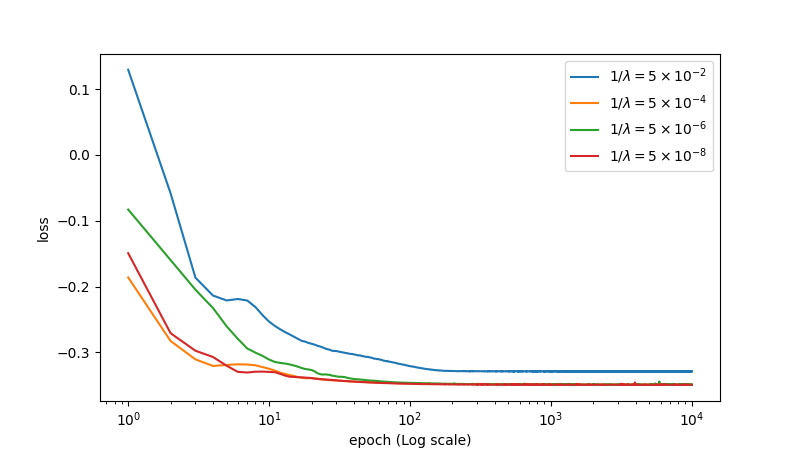}
\caption{The learning curves for $1/\lambda=5\times 10^{-2}, 5\times 10^{-4}, 5\times 10^{-6}, 5\times 10^{-8}$, in the case where
$\mu_0$ and $\mu_1$ are given by the dataset distributions of two circles and one of a double crescent, respectively. }
\label{fig:4.4}
\end{figure}
As opposed to the previous 1D example, the learning curves exhibit stable behaviors and a decreasing property with respect to $\lambda$.

Figures \ref{fig:4.5} and \ref{fig:4.6} exhibit results of a numerical test similar to the above one, in the case where two datasets have higher dispersions.
Here we set $b=0$ and $\sigma=0.1I_2$.
\begin{figure}[htbp]
\centering
\includegraphics[width=1.0\columnwidth, bb = 0 0 1152 216]{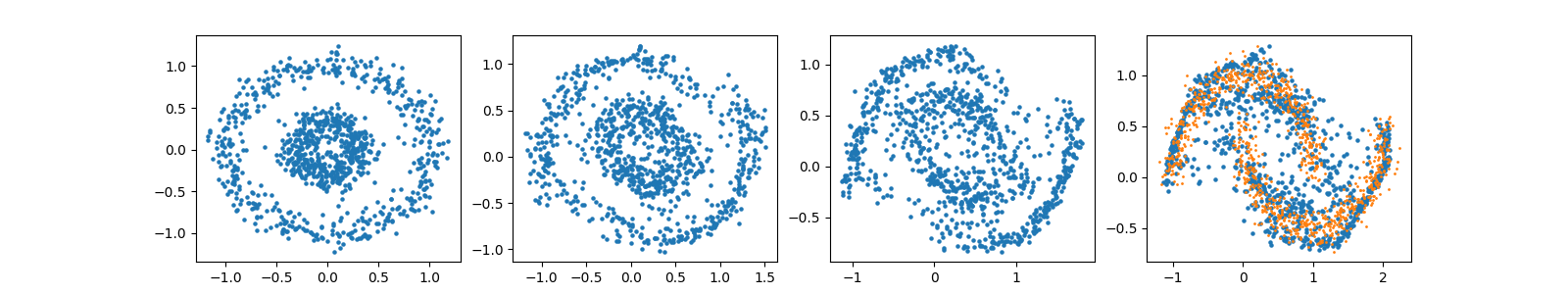}
\caption{Interpolation of two dataset distributions of two circles and one of a double crescent, with different dispersion level.
The case of $1/\lambda=5\times 10^{-3}$ and $5000$ epochs.
The scatter plot of optimized $X_t^{(\theta)}$'s for $t=0, 0.33, 0.67, 1$
from left to right. The target distribution is plotted in orange.}
\label{fig:4.5}
\end{figure}
\begin{figure}[htbp]
\centering
\includegraphics[width=0.6\columnwidth, bb = 0 0 576 324]{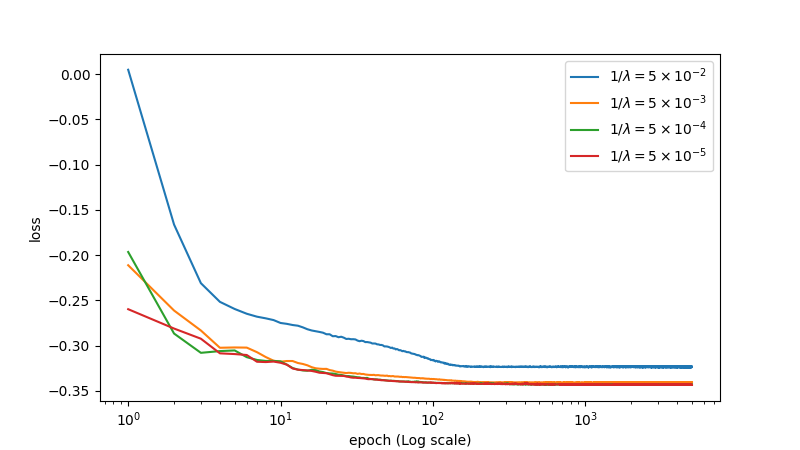}
\caption{The learning curves for $1/\lambda=5\times 10^{-2}, 5\times 10^{-3}, 5\times 10^{-4}, 5\times 10^{-5}$, in the case where
$\mu_0$ and $\mu_1$ are given by the dataset distributions presented in Figure \ref{fig:4.5}.}
\label{fig:4.6}
\end{figure}

\section{Proof of Theorem \ref{thm:3.1}}\label{sec:5}

This section is devoted to a proof of Theorem \ref{thm:3.1}. Then assume that the conditions (A1)--(A5) always hold.

First, we need to enlarge the set of control policies as follows:
\begin{defn}
We say that a quadruple $\pi = (\mathbb{Q}, B, u, Y)$ is an \textit{admissible control system} if
\begin{enumerate}[label=(\roman*)]
\item $\mathbb{Q}$ is a probability measure on $(\Omega,\mathcal{F})$ such that $\mathbb{Q}\sim \mathbb{P}$;
\item $B=\{B_t\}_{0\le t\le 1}$ is a $m$-dimensional $\mathbb{F}$-Brownian motion on $(\Omega,\mathcal{F},\mathbb{Q})$;
\item $u=\{u_t\}_{0\le t\le 1}$ is an $\mathbb{R}^m$-valued $\mathbb{F}$-progressively measurable process such that
\[
 \mathbb{E}_{\mathbb{Q}}\left[\int_0^1|u_s|^2ds\right]<\infty,
\]
\[
 \mathbb{E}_{\mathbb{Q}}\left[\exp\left(-\int_0^1u_s^{\mathsf{T}}dB_s - \frac{1}{2}\int_0^1|u_s|^2ds\right)\right]=1;
\]
\item $Y=\{Y_t\}_{0\le t\le 1}$ is an $\mathbb{R}^d$-valued continuous and $\mathbb{F}$-adapted process satisfying
\[
 Y_t= Y_0 + \int_0^t(b(s,Y_s) + \sigma(s,Y_s)u_s)ds + \int_0^1\sigma(s,Y_s) dB_s, \quad 0\le t\le 1,
\]
and $\mathbb{Q}Y_0^{-1}=\mu_0$.
\end{enumerate}
We write $\Pi$ for the set of all admissible control systems.
\end{defn}

It should be noted that $(\mathbb{Q}^u, B^u, \{u(t,X_t)\}_{0\le t\le 1}, X)\in \Pi$ for any $u\in\mathcal{U}$.
Further, denote by $\Pi_1$ the set of all $\pi=(\mathbb{Q},B,u,Y)\in\Pi$
such that $\mathbb{Q}Y_1^{-1}=\mu_1$.

For any admissible control system $\pi=(\mathbb{Q}, W, u, Y)\in\Pi_1$,
define the control criterion by
\[
 J(\pi):=\mathbb{E}_{\mathbb{Q}}\left[\int_0^1|u_s|^2ds\right].
\]
Then consider the optimal stochastic control problem
\begin{equation}
\tag{C}
 J^*:=\inf_{\pi\in\Pi_1} J(\pi).
\end{equation}

The next lemma describes a connection between the problems (S) and (C), which is a generalization of Lemma 2.6 in \cite{fol:1985}.
\begin{lem}
\label{lem:5.2}
Let $Q\in\mathcal{P}(\mathbb{W}^d)$ satisfies $Q_0=\mu_0$, $Q_1=\mu_1$, and $H(Q|P)<\infty$.
Then there exists $\pi\in\Pi_1$ such that
\begin{equation}
\label{eq:5.3}
 H(Q|P)= \frac{1}{2}J(\pi).
\end{equation}
\end{lem}
\begin{proof}
Let $Q\in\mathcal{P}(\mathbb{W}^d)$ such that $Q_0=\mu_0$, $Q_1=\mu_1$, and $H(Q|P)<\infty$.
Define the probability measure $\mathbb{Q}$ on $(\Omega,\mathcal{F})$ by
\[
 \frac{d\mathbb{Q}}{d\mathbb{P}} = \frac{dQ}{dP}(X).
\]
Let $\{\mathcal{G}_t\}$ be the natural filtration generated by $W$, augmented with $\mathbb{P}$.
Then by the martingale representation theorem,
\[
 \kappa_t:= \mathbb{E}\left[\left.\frac{d\mathbb{Q}}{d\mathbb{P}}\right| \mathcal{G}_t\right]
 = 1 + \int_0^t\phi_s^{\mathsf{T}}dW_s, \quad \mathbb{P}\text{-a.s.}, \;\; 0\le t\le 1,
\]
for some $\mathbb{R}^d$-valued and $\{\mathcal{G}_t\}$-progressively measurable process $\phi$ satisfying
\[
 \int_0^1|\phi_t|^2dt < \infty, \quad \mathbb{P}\text{-a.s.}
\]
Put $u_t:=\kappa_t^{-1}1_{\{\kappa_t>0\}}\phi_t$, $0\le t\le 1$.
Since the process $\kappa_t$ is a nonnegative
$(\{\mathcal{G}_t\},\mathbb{P})$-supermartingale, we apply a generalized Girsanov-Maruyama theorem
(see Theorem 6.2 in Liptser and Shiryaev \cite{lip-shi:2001}) to find that the process
\[
 B_t: = W_t - \int_0^t u_sds, \quad 0\le t\le 1,
\]
is an $(\{\mathcal{G}_t\}, \mathbb{Q})$-Brownian motion. By the same argument as in Section 6.3 in \cite{lip-shi:2001},
we obtain
\begin{equation}
\label{eq:5.4}
 \int_0^1|u_t|^2dt < \infty, \quad \mathbb{Q}\text{-a.s.}
\end{equation}
and
\begin{equation}
\label{eq:5.5.1}
 \kappa_t = \exp\left(\int_0^tu_s^{\mathsf{T}}dB_s + \frac{1}{2}\int_0^t|u_s|^2ds\right), \quad \mathbb{Q}\text{-a.s.}, \;\; 0\le t\le 1.
\end{equation}
For the reader's convenience, we shall describe detailed proofs of \eqref{eq:5.4} and \eqref{eq:5.5.1}. To this end, observe
\[
 \mathbb{P}\left(\int_0^1|\kappa_tu_t|^2dt<\infty\right) = \mathbb{P}\left(\int_0^1|\phi_t|^21_{\{\kappa_t>0\}}dt<\infty\right)
 \ge \mathbb{P}\left(\int_0^1|\phi_t|^2 dt<\infty\right) = 1.
\]
Further, use Lemma 6.5 in \cite{lip-shi:2001} to get $\mathbb{Q}(\inf_{0\le t\le 1}\kappa_t=0)=0$.
Since $\kappa_t$ is continuous $\mathbb{Q}$-a.s., we get \eqref{eq:5.4}.
Hence  the stochastic integrals
$\int_0^t\kappa_su_s^{\mathsf{T}}dW_s$, $\int_0^t \kappa_su_s^{\mathsf{T}}dB_s$, and
$\int_0^tu_s^{\mathsf{T}}dB_s$ are well-defined. From this,
\[
 \kappa_t = 1 + \int_0^t\kappa_s u_s^{\mathsf{T}} dW_s = 1 + \int_0^t\kappa_su_s^{\mathsf{T}}dB_s + \int_0^t\kappa_s|u_s|^2ds,
 \quad \mathbb{Q}\text{-a.s.},
\]
whence by It{\^o} formula, \eqref{eq:5.5.1} holds. Moreover, as in the proof of Lemma 2.6 in F{\"o}llmer \cite{fol:1985}, we
can show that
\begin{equation}
\label{eq:5.6.1}
 \mathbb{E}_{\mathbb{Q}}\int_0^1 |u_t|^2 dt < \infty.
\end{equation}
Indeed, put $\tau_n=\inf\{t>0: \int_0^t|u_s|^2ds > n\}\wedge 1$, and
\[
 Z_n = \exp\left(\int_0^{\tau_n}u_s^{\mathsf{T}}dB_s + \frac{1}{2}\int_0^{\tau_n}|u_s|^2ds\right).
\]
We can define the probability measure $\mathbb{Q}_n$ on $(\Omega,\mathcal{G}_1)$ by
$d\mathbb{Q}_n/d\mathbb{P}=Z_n$. Then,
\[
 \mathbb{E}_{\mathbb{Q}}\left[\log \frac{d\mathbb{Q}}{d\mathbb{P}}\right] =
 \mathbb{E}_{\mathbb{Q}}\left[\log \frac{d\mathbb{Q}}{d\mathbb{Q}_n}\right] + \mathbb{E}_{\mathbb{Q}}[\log Z_n] \\
 \ge \mathbb{E}_{\mathbb{Q}}[\log Z_n] = \frac{1}{2}\mathbb{E}_{\mathbb{Q}}\int_0^{\tau_n}|u_s|^2ds.
\]
Use the monotone convergence theorem to obtain
\[
 \mathbb{E}_{\mathbb{Q}}\int_0^1|u_s|^2ds = \lim_{n\to\infty}\mathbb{E}_{\mathbb{Q}}\int_0^{\tau_n}|u_s|^2ds \le
 2\mathbb{E}_{\mathbb{Q}}\left[\log \frac{d\mathbb{Q}}{d\mathbb{P}}\right]= 2H(Q|P).
\]
Thus \eqref{eq:5.6.1} follows.
On the other hand, using Jensen's inequality for the conditional expectation, we obtain
\[
 H(Q|P) \le \mathbb{E}_{\mathbb{Q}}\left[\log\kappa_1\right]
  = \frac{1}{2}\mathbb{E}_{\mathbb{Q}}\int_0^1|u_t|^2dt.
\]
Hence
\[
 H(Q|P) = \frac{1}{2}\mathbb{E}_{\mathbb{Q}}\int_0^1|u_t|^2dt.
\]
Since $W_t$ is an It{\^o} process under $\mathbb{Q}$, the representation
\[
 X_t = X_0 + \int_0^t (b(s,X_s) + \sigma(s,X_s)u_s)ds+ \int_0^t \sigma(s,X_s)dB_s, \quad \mathbb{Q}\text{-a.s.}
\]
holds. Moreover, we have
\[
 \mathbb{E}_{\mathbb{Q}}\left[\exp\left(-\int_0^1u_t^{\mathsf{T}}dB_t-\frac{1}{2}\int_0^1|u_t|^2dt\right)\right]
 =\mathbb{E}\left[\kappa_1\exp\left(-\int_0^1u_t^{\mathsf{T}}dB_t-\frac{1}{2}\int_0^1|u_t|^2dt\right)\right]=1,
\]
whence $(\mathbb{Q},B,u,X)\in\Pi_1$ and \eqref{eq:5.3} hold.
\end{proof}

By Lemma \ref{lem:5.2}, we have
\[
 H^*\ge \frac{1}{2}J^*.
\]
By (A2) and $\mu_1^*\sim\mu_1$, the density $\varphi^*_1(y):=d\mu_1^*/dy$ is positive. 
So it follows from Theorem 2 in \cite{jam:1975} that the function 
\[
 h(t,x) = \mathbb{E}_{\mathbb{P}}[\varphi_1^*(X_1^{t,x})], \quad 0\le t\le 1,\;\; x\in\mathbb{R}^d,
\]
is in $C^{1,2}([0,1)\times\mathbb{R}^d)$ and satisfies
\begin{align*}
 \frac{\partial}{\partial t}  h(t,y) + b(t,x)^{\mathsf{T}}D_xh(t,x) + \frac{1}{2}\mathrm{tr}(\sigma(t,x)\sigma(t,x)^{\mathsf{T}}D_x^2h(t,x)) &= 0, \quad 0\le t< 1,\;\; x\in\mathbb{R}^d, \\
 h(1,x) &= \varphi_1^*(x), \quad x\in\mathbb{R}^d,
\end{align*}
where $\{X_s^{t,x}\}$ is the unique solution of \eqref{eq:3.0} with initial condition replaced by $(t,x)$, i,e.,
\[
 X_s^{t,x} = x + \int_t^sb(r,X_r^{t,x})dr + \int_t^s \sigma(r,X_r^{t,x})dW_r, \quad t\le s\le 1,
\]
and $D_xh$ and $D_x^2h$ stand for the gradient and Hessian of $h$ with respect to $x$, respectively.
Then, by \eqref{eq:3.2},
\[
 h(0,x)=\int_{\mathbb{R}^d}p(0,x,1,y)\mu_1^*(dy)= \frac{d\mu_0}{d\mu_0^*}(x), \quad x\in\mathbb{R}^d,
\]
and so
\[
 \mathbb{E}\left[\frac{h(1,X_1)}{h(0,X_0)}\right]=\int_{\mathbb{R}^d}\frac{1}{h(0,x)}\int_{\mathbb{R}^d}p(0,x,1,y)\mu_1^*(dy)\mu_0(dx)=1.
\]
From this
we can define the probability measure $\mathbb{P}^*$ by
\[
 \frac{d\mathbb{P}^*}{d\mathbb{P}}= \frac{h(1,X_1)}{h(0,X_0)}.
\]
Then consider the function $u^*(t,x):=\sigma(t,x)^{\mathsf{T}}D_x\log h(t,x)$, $(t,x)\in [0,1]\times\mathbb{R}^d$ and 
the process $W^*_t:=W_t - \int_0^tu^*_sds$, $0\le t\le 1$. 
Further, $P^*:=\mathbb{P}^*X^{-1}$ and $\pi^*:=(\mathbb{P}^*,W^*,\{u^*(t,X_t)\}_{0\le t\le 1},X)$. 
\begin{thm}
\label{thm:5.3}
We have $u^*\in \mathcal{U}$, $\mathbb{Q}^{u^*} = P^*$, and $B^{u^*} = W^*$. Moreover, $\pi^*\in \Pi_1$  
and is optimal for the problem (C). Furthermore, it holds that
\[
   H(P^*|P)=\tfrac12 J(\pi^*) =
    \int_{\mathbb{R}^d}\log h(1,y) \mu_1(dy)
    - \int_{\mathbb{R}^d}\log h(0,x) \mu_0(dx) < \infty.
\]
In particular, $P^*$ is the unique solution of the problem (S).
\end{thm}
\begin{proof}
Step (i).
First we will prove that $u^*\in\mathcal{U}$, $\pi^*\in\Pi_1$, and $H(P^*|P)<\infty$.
Since the process $\mathbb{E}[h(1,X_1)/h(0,X_0)|\mathcal{F}_t]$ is a $\mathbb{P}$-martingale with unit initial value, we have
\[
 \mathbb{P}^*(X_0\in A) = \mathbb{E}\left[\left.\mathbb{E}\left[\frac{h(1,X_1)}{h(0,X_0)}\,\right|\,\mathcal{F}_0\right]1_{\{X_0\in A\}}\right] =
 \mathbb{P}(X_0\in A) = \mu_0(A),
\]
and also by \eqref{eq:3.2},  
\begin{align*}
    \mathbb{P}^*(X_1\in A)
    &= \mathbb{E}_{\mathbb{P}}
       \!\left[
         1_{\{X_1\in A\}}\,
         \frac{h(1,X_1)}{h(0,X_0)}
       \right]
    = \int_{\mathbb{R}^d} \frac{d\mu^*_0}{d\mu_0}(x)
      \int_A \frac{d\mu^*_1}{dy}(y)p(0,x,1,y)dy\mu_0(dx) \\
    &= \int_A \int_{\mathbb{R}^d} \frac{d\mu^*_0}{d\mu_0}(x)
       p(0,x,1,y)\mu_0(dx)\,d\mu^*_1(y)\\
    &= \int_A \int_{\mathbb{R}^d} p(0,x,1,y)d\mu^*_0(x)d\mu^*_1(y)
     = \mu_1(A).
  \end{align*}
for any $A\in\mathcal{B}(\mathbb{R}^d)$. 

Put 
\[
\nu(E) = \int_E p(0,x,1,y)\mu_0^*(dx)\mu_1^*(dy), \quad E\in\mathcal{B}(\mathbb{R}^d\times\mathbb{R}^d). 
\]
Then, it is known that under (A2), 
\begin{align*}
\int_{\mathbb{R}^d}\log h(1,y)\mu_1(dy) - \int_{\mathbb{R}^d}\log h(0,x)\mu_0(dx) 
&= H(\nu\,|\, p(0,x,1,y)\mu_0(dx)dy) \\ 
&= \inf H(\nu^{\prime}\,|\, p(0,x,1,y)\mu_0(dx)dy) < \infty, 
\end{align*}
where the infimum is taken over all $\nu\in\mathcal{P}(\mathbb{R}^d\times\mathbb{R}^d)$ such that 
$\nu^{\prime}(dx\times\mathbb{R}^d)=\mu_0(dx)$ and $\nu^{\prime}(\mathbb{R}^d\times dy)=\mu_1(dy)$. 
See, e.g., Theorem 2.1 in Nutz \cite{nut:2022}. 

Note that
\[
 dh(t,X_t) = h(t,X_t)(u^*_t)^{\mathsf{T}}dW_t,
\]
where $u_t^*=u^*(t,X_t)$, and consider the stopping times
$\tau_n:=\inf\{t>0; |u_t^*|>n\}$, $n\in\mathbb{N}$.
Then, for each $n$ define the probability measure $\mathbb{P}_n$ by
\begin{align*}
 \frac{d\mathbb{P}_n}{d\mathbb{P}} :&= \frac{h(1\wedge\tau_n,X_{1\wedge\tau_n})}{h(0,X_0)}
 = \exp\left[\int_0^1(\psi^{(n)}_t)^{\mathsf{T}}dW_t - \frac{1}{2}\int_0^1|\psi_t^{(n)}|^2dt\right] \\
 &= \exp\left[\int_0^1(\psi^{(n)}_t)^{\mathsf{T}}dW^*_t
 +\frac{1}{2}\int_0^1|\psi_t^{(n)}|^2dt\right],
\end{align*}
where $\psi_t^{(n)} = u_t^*1_{\{t\le\tau_n\}}$. By the monotone convergence theorem,
\begin{equation}
\label{eq:5.5}
\mathbb{E}_{\mathbb{P}^*}\int_0^1|u_t^*|^2dt
= \lim_{n\to\infty}\mathbb{E}_{\mathbb{P}^*}\left[\int_0^1|\psi_t^{(n)}|^2dt\right]
= 2\lim_{n\to\infty}\mathbb{E}_{\mathbb{P}^*}\left[\log\frac{d\mathbb{P}_n}{d\mathbb{P}}\right].
\end{equation}
On the other hand, since the relative entropy is nonnegative, we obtain
\begin{align*}
\mathbb{E}_{\mathbb{P}^*}\left[\log\frac{d\mathbb{P}_n}{d\mathbb{P}}\right]
&\le \mathbb{E}_{\mathbb{P}^*}\left[\log\frac{d\mathbb{P}_n}{d\mathbb{P}}\right]
    + \mathbb{E}_{\mathbb{P}^*}\left[\log\frac{d\mathbb{P}^*}{d\mathbb{P}_n}\right]
  = \mathbb{E}_{\mathbb{P}^*}\left[\log\frac{d\mathbb{P}^*}{d\mathbb{P}}\right]  \\
&= \int_{\mathbb{R}^d}\log h(1,y)\mu_1(dy) - \int_{\mathbb{R}^d}\log h(0,x)\mu_0(dx) < \infty.
\end{align*}
From this and \eqref{eq:5.5} we have 
\[
 \mathbb{E}_{\mathbb{P}^*}\int_0^1|u_t^*|^2dt<\infty, 
\]
whence $\int_0^1|u_t^*|^2dt<\infty$, $\mathbb{P}$-a.s. 
Thus,
\[
 \frac{d\mathbb{P}^*}{d\mathbb{P}}= \exp\left[\int_0^1(u_t^*)^{\mathsf{T}}dW_t
 - \frac{1}{2}\int_0^1|u_t^*|^2dt\right],
\]
and so by the Girsanov-Maruyama theorem, $\{W^*_t\}$ is an $\mathbb{F}$-Brownian motion under $\mathbb{P}^*$.
This means $\mathbb{P}^*=\mathbb{Q}^{u^*}$ and $W^*=B^{u^*}$.
Hence $u^*\in\mathcal{U}$ and $\pi^*\in\Pi_1$. 
Then, the above arguments show that
\begin{equation}
\label{eq:5.6}
 H(P^*|P) = \mathbb{E}_{\mathbb{P}^*}\left[\log\frac{d\mathbb{P}^*}{d\mathbb{P}}\right] =
 \frac{1}{2}J(\pi^*)
 = \int_{\mathbb{R}^d}\log h(1,y)\mu_1(dy) - \int_{\mathbb{R}^d}\log h(0,x)\mu_0(dx) < \infty.
\end{equation}

Step (ii). We will prove the optimality of $\pi^*$. Let $\pi=(\mathbb{Q},B,u,Y)\in\Pi_1$ be arbitrary.
Then the process
\[
 \hat{W}_t := B_t + \int_0^tu_sds, \quad 0\le t\le 1,
\]
is an $m$-dimensional $\mathbb{F}$-Brownian motion under $\hat{\mathbb{P}}$ defined by
\[
 \frac{d\hat{\mathbb{P}}}{d\mathbb{Q}}
 = \exp\left[-\int_0^1u_t^{\mathsf{T}}dB_t - \frac{1}{2}\int_0^1|u_t|^2dt\right].
\]
Since the controlled process $Y_t$ satisfies
\[
 dY_t= b(t,Y_t)dt + \sigma(t,Y_t)d\hat{W}_t 
\]
and $\hat{W}$ is a Brownian motion under $\hat{\mathbb{P}}$, the distribution of $Y$ under $\hat{P}$ is the same as that of $X$ under $\mathbb{P}$. 
Hence,
\begin{align*}
1 &= \mathbb{E}\left[e^{\log h(1,X_1)-\log h(0,X_0)}\right]
   = \mathbb{E}_{\hat{\mathbb{P}}}\left[e^{\log h(1,Y_1)-\log h(0,Y_0)}\right] \\
   &= \mathbb{E}_{\mathbb{Q}}\left[\exp\left\{\log h(1,Y_1)-\log h(0,Y_0)
         -\int_0^1u_t^{\mathsf{T}}dB_t - \frac{1}{2}\int_0^1|u_t|^2dt\right\}\right]  \\
  &\ge \exp\left\{\mathbb{E}_{\mathbb{Q}}\left[\log h (1,Y_1)-\log h(0,Y_0)
          -\int_0^1u_t^{\mathsf{T}}dB_t - \frac{1}{2}\int_0^1|u_t|^2dt\right]\right\},
\end{align*}
where we have used Jensen's inequality in the last inequality. Therefore,
\begin{align*}
 \frac{1}{2}\mathbb{E}_{\mathbb{Q}}\left[\int_0^1|u_t|^2dt\right]
 &\ge \mathbb{E}_{\mathbb{Q}}\left[\log h (1,Y_1)-\log h(0,Y_0)\right]
 =\int_{\mathbb{R}^d}\log h (1,y)\mu_1(dy) - \int_{\mathbb{R}^d}\log h (0,x)\mu_0(dx) \\
 &= \frac{1}{2}J(\pi^{\ast}).
\end{align*}
Consequently, we deduce that $\pi^*$ is an optimal solution to (C). 
Combining this with \eqref{eq:5.6} and Lemma \ref{lem:5.2}, we conclude that $P^*$ is a solution of (S). 
Moreover, the uniqueness of the solution to (P) follows from the strict convexity of the relative entropy.
\end{proof}
\begin{proof}[Proof of Theorem $\ref{thm:3.1}$]

Step (i). To prove \eqref{eq:3.5}, we will use an argument similar to that in the proof of Lemma 1 in \cite[Section 10.11]{lue:1997},
which is presented in the framework of penalty function methods in the constrained optimization.
Assume that
\[
 \limsup_{n\to\infty}\lambda_{n}\gamma(\mathbb{Q}^{u_n}X_1^{-1},\mu_1)^2 = 5\delta
\]
holds for some $\delta>0$. Then there exists a subsequence $\{n_k\}$ such that
\[
 \lim_{k\to\infty}\lambda_{n_k}\gamma(\mathbb{Q}^{u_{n_k}}X_1^{-1},\mu_1)^2 = 5\delta.
\]
Theorem \ref{thm:5.3} means that for any $n\in\mathbb{N}$,
\begin{equation}
\label{eq:5.7}
 J_{\lambda_n}(u_n)\le J^*_{\lambda_n} + \varepsilon_n\le J(\pi^*) + \varepsilon_n = 2H^* + \varepsilon_n, 
\end{equation}
where the second inequality in \eqref{eq:5.7} follows from the facts
  that $u^*\in \mathcal{U}$ and $\gamma(\mathbb{Q} u^*(X_1)^{-1},\mu_1)=0$.
In particular, the sequence $\{J_{\lambda_{n_k}}(u_{n_k})\}_{k=1}^{\infty}$ is bounded, whence
we can take a further subsequence $\{n_{k_m}\}$ such that
\[
 \lim_{m\to\infty}J_{\bar{\lambda}_m}(\bar{u}_m) = \kappa:=\limsup_{k\to\infty}J_{\lambda_{n_k}}(u_{n_k})<\infty,
\]
where $\bar{\lambda}_m=\lambda_{n_{k_m}}$ and $\bar{u}_m=u_{n_{k_m}}$.
Then put $\bar{\gamma}_m=\gamma(\mathbb{Q}^{\bar{u}_m}X_1^{-1},\mu_1)$.

Now choose $m_0$ such that $\kappa<J_{\bar{\lambda}_{m_0}}(\bar{u}_{m_0})+ \delta$, and then select $m_1$ such that 
$J_{\bar{\lambda}_{m_1}}(\bar{u}_{m_1})<\kappa +\delta$, 
$\bar{\lambda}_{m_1}>7\bar{\lambda}_{m_0}$ and $3\delta +\bar{\varepsilon}_{m_0} < \bar{\lambda}_{m_1}\bar{\gamma}_{m_1}^2 <7\delta$, 
where $\bar{\varepsilon}_m=\varepsilon_{n_{k_m}}$. 
With these choices, we get
\begin{align*}
 \kappa &< J_{\bar{\lambda}_{m_0}}(\bar{u}_{m_0}) + \delta
 \le J^*_{\bar{\lambda}_{m_0}} +\bar{\varepsilon}_{m_0} + \delta
 \le J_{\bar{\lambda}_{m_0}}(\bar{u}_{m_1}) +\bar{\varepsilon}_{m_0} + \delta \\
 &= \mathbb{E}_{\mathbb{Q}^{\bar{u}_{m_1}}}\int_0^1|\bar{u}_{m_1}(t)|^2dt + \frac{\bar{\lambda}_{m_0}}{\bar{\lambda}_{m_1}}\bar{\lambda}_{m_1}
       \bar{\gamma}_{m_1}^2 + \bar{\varepsilon}_{m_0} + \delta
 < \mathbb{E}_{\mathbb{Q}^{\bar{u}_{m_1}}}\int_0^1|\bar{u}_{m_1}(t)|^2dt + \frac{\bar{\lambda}_{m_1}}{7}\bar{\gamma}_{m_1}^2 +\bar{\varepsilon}_{m_0} + \delta \\
 &< \mathbb{E}_{\mathbb{Q}^{\bar{u}_{m_1}}}\int_0^1|\bar{u}_{m_1}(t)|^2dt + 2\delta +\bar{\varepsilon}_{m_0}
 < \mathbb{E}_{\mathbb{Q}^{\bar{u}_{m_1}}}\int_0^1|\bar{u}_{m_1}(t)|^2dt + \bar{\lambda}_{m_1}\bar{\gamma}_{m_1}^2 - \delta
   = J_{\bar{\lambda}_{m_1}}(\bar{u}_{m_1}) -\delta
 < \kappa,
\end{align*}
which is impossible. Thus \eqref{eq:3.5} follows.

Step (ii).
Let us prove the tightness of the sequence of $Q_n:=\mathbb{Q}^{u_n}(X_1)^{-1}$, $n\in\mathbb{N}$.
From Theorem 7.3 in Billingsley \cite{bil:1999} it is sufficient to show that for any $\delta>0$,
\begin{equation}
\label{eq:5.9}
 \lim_{h\to 0}\sup_{n\ge 1}\mathbb{Q}^{u_n}\left(\max_{\substack{0\le s<t\le 1\\ 0<t-s<h}}|X_t-X_s|>\delta\right)=0.
\end{equation}
To this end, consider the process
\[
 Y_n(t) := X_0 + \int_0^t\left[b(s,X_s) + \sigma(s,X_s) u_n(s)\right]ds, \quad 0\le t\le 1.
\]
This process satisfies
\begin{align*}
 \sup_{n}\mathbb{Q}^{u_n}\left(\max_{\substack{0\le s<t\le 1\\ 0<t-s<h}}|Y_n(t)-Y_n(s)|>\delta/2\right)
 &\le \frac{4}{\delta^2}\sup_n\mathbb{E}_{\mathbb{Q}^{u_n}}\max_{\substack{0\le s<t\le 1\\ 0<t-s<h}}|Y_n(t)-Y_n(s)|^2 \\
 &\le\frac{Ch}{\delta^2}\sup_n\mathbb{E}_{\mathbb{Q}^{u_n}}\max_{\substack{0\le s<t\le 1\\ 0<t-s<h}}\int_s^t\left[1+|u_{n}(r)|^2\right]dr \\
 &\le\frac{Ch}{\delta^2}\left[1 + \sup_n\mathbb{E}_{\mathbb{Q}^{u_n}}\int_0^1|u_{n}(r)|^2dr\right].
\end{align*}
For a fixed $N\ge 1$ take $\{t_i\}$ so that $0=t_0<t_1<\cdots t_N = 1$ and $t_{\ell}-t_{\ell -1}=1/N$. Then from Theorem 7.4 in Billingsley \cite{bil:1999} it follows that 
\begin{align*}
  \mathbb{Q}^{u_n}\left(\max_{\substack{0\le s<t\le 1\\ 0<t-s<1/N}}\left|\int_s^t\sigma(r,X_r)dB_r^{u_n}\right|>\delta/2 \right) 
 &\le \sum_{\ell=1}^N\mathbb{Q}^{u_n}\left(\sup_{t_{\ell -1}\le s\le t_{\ell}}\left|\int_{t_{\ell-1}}^s\sigma(r,X_r)dB_r^{u_n}\right|> \delta/6\right) \\ 
 &\le \frac{C}{\delta^4}\sum_{\ell=1}^N\mathbb{E}_{\mathbb{Q}^{u_n}}\left|\int_{t_{\ell-1}}^{t_{\ell}}|\sigma(r,X_r)|^2ds\right|^2 \\ 
 &\le \frac{C}{\delta^4N}, 
\end{align*}
where $C>0$ is a constant, and we have used Chebyshef's inequality, the Burkholder-Davis-Gundy inequality, and the boundedness of $\sigma$. 
So we have 
\[
 \lim_{h\to 0}\sup_{n\ge 1}\mathbb{Q}^{u_n}\left(\max_{\substack{0\le s<t\le 1\\ 0<t-s<h}}\left|\int_s^t \sigma(r,X_r)dB_r^{u_n}\right|>\delta/2\right)=0.
\]
Hence using
\begin{align*}
 &\mathbb{Q}^{u_n}\left(\max_{\substack{0\le s<t\le 1\\ 0<t-s<h}}|X_t-X_s|>\delta\right) \\
 &\le \mathbb{Q}^{u_n}\left(\max_{\substack{0\le s<t\le 1\\ 0<t-s<h}}|Y_n(t)-Y_n(s)|>\delta/2\right)
     + \mathbb{Q}^{u_n}\left(\max_{\substack{0\le s<t\le 1\\ 0<t-s<h}}\left|\int_s^t \sigma(r,X_r)dB_r^{u_n}\right|>\delta/2\right),
\end{align*}
we derive \eqref{eq:5.9}.
Therefore the sequence $\{Q_{n}\}$ is tight, i.e.,
there exist a subsequence $\{n_k\}$ and $\hat{P}$ such that $Q_{n_k}$ weakly converges to $\hat{P}$.
From this and Proposition \ref{prop:2.1} we obtain
\[
 \gamma(\hat{P}_1,\mu_1)\le \gamma(\hat{P}_1, (Q_{n_k})_1)
  + \gamma(\mathbb{Q}^{u_{n_k}}(X_1)^{-1},\mu_1)\to 0,
\]
as $k\to\infty$, whence $\hat{P}_1=\mu_1$.
Further, by the lower semi-continuity of $Q\mapsto H(Q|P)$, which follows from the convexity of $x\mapsto x\log x - x +1$ and Fatou's lemma,
\begin{align*}
 H^*\le H(\hat{P} | P) \le\liminf_{k\to\infty} H(Q_{n_k}|P).
\end{align*}
Further, in view of \eqref{eq:5.7}, as in the proof of Theorem \ref{thm:5.3} we have $2H(Q_{n_k}|P)=J(u_{n_k})\le J_{\lambda_{n_k}}(u_{n_k})$. So, 
\[
 H^*\le H(\hat{P} | P)\le \frac{1}{2}\liminf_{k\to\infty}J_{\lambda_{n_k}}(u_{n_k})\le H^*. 
\]
This means that $\hat{P}$ is an optimal solution to (S), whence by uniqueness, we obtain $\hat{P}=P^*$.
What we have shown now is that each subsequence $\{Q_{n_k}\}$ contain a further subsequence
$\{Q_{n_{k_j}}\}$ converging weakly to $P^*$.
Applying Theorem 2.6 in \cite{bil:1999}, we deduce that $\{Q_n\}$ converges weakly to $P^*$ as $n\to\infty$.
Again by the lower semi-continuity of the relative entropy,
\[
 \liminf_{n\to\infty}J_{\lambda_n}(u_n)\ge 2\liminf_{n\to\infty}H(Q_n|P) \ge 2H^*\ge \limsup_{n\to\infty} J_{\lambda_n}(u_n).
\]
This shows \eqref{eq:3.6}.

Step (iii). Since $u^*\in\mathcal{U}$, it follows from Theorem \ref{thm:5.3} that
\[
 2H^*=J^*=J_{\lambda}(u^*)\ge \inf_{u\in\mathcal{U}}J_{\lambda}(u), \quad \lambda>0,
\]
whence
\[
 H^*\ge \frac{1}{2}\sup_{\lambda>0}\inf_{u\in\mathcal{U}}J_{\lambda}(u).
\]

On the other hand, by the definition of $u_n$,
\[
 J_{\lambda_n}(u_n)\le \sup_{\lambda>0}\inf_{u\in\mathcal{U}}J_{\lambda}(u) + \varepsilon_n.
\]
Taking the limit and using \eqref{eq:3.6},  we get
\[
 H^* \le \frac{1}{2}\sup_{\lambda>0}\inf_{u\in\mathcal{U}}J_{\lambda}(u),
\]
which completes the proof of the theorem.
\end{proof}

\subsection*{Acknowledgements}

This study is supported by JSPS KAKENHI Grant Number JP21K03364.

\bibliographystyle{hplain}
\bibliography{../mybib}

\end{document}